\documentclass{amsart}
\usepackage{amssymb}
\usepackage{amsfonts}
\usepackage{amssymb}
\usepackage{amsmath}
\usepackage{amsthm}
\usepackage{enumerate}
\usepackage{tabularx}
\usepackage{centernot}
\usepackage{mathtools}
\usepackage{amsthm,amssymb}

\usepackage{tikz}
\usetikzlibrary{decorations.pathmorphing,shapes}
\usetikzlibrary{trees}
\usepackage{bussproofs}

\def\prepQuaternary{%
    \ifnum\theLevel<4
        \errmessage{Hypotheses missing!}
    \fi%
    \edef\rrcurBox{\thecur{myBox}}
    \edef\rrcurScoreStart{\thecur{myScoreStart}}%
    \edef\rrcurCenter{\thecur{myCenter}}%
    \edef\rrcurScoreEnd{\thecur{myScoreEnd}}%
    \advance\theLevel by-1%
    \edef\rcurBox{\thecur{myBox}}
    \edef\rcurScoreStart{\thecur{myScoreStart}}%
    \edef\rcurCenter{\thecur{myCenter}}%
    \edef\rcurScoreEnd{\thecur{myScoreEnd}}%
    \advance\theLevel by-1%
    \edef\ccurBox{\thecur{myBox}}
    \edef\ccurScoreStart{\thecur{myScoreStart}}%
    \edef\ccurCenter{\thecur{myCenter}}%
    \edef\ccurScoreEnd{\thecur{myScoreEnd}}%
    \advance\theLevel by-1%
    \edef\lcurBox{\thecur{myBox}}
    \edef\lcurScoreStart{\thecur{myScoreStart}}%
    \edef\lcurCenter{\thecur{myCenter}}%
    \edef\lcurScoreEnd{\thecur{myScoreEnd}}%
}

\def\QuaternaryInf$#1\fCenter#2${%
    \prepQuaternary%
    \buildConclusion{#1}{#2}%
    \joinQuaternary%
    \resetInferenceDefaults%
    \ignorespaces%
}

\def\QuaternaryInfC#1{%
    \prepQuaternary%
    \buildConclusionC{#1}%
    \joinQuaternary%
    \resetInferenceDefaults%
    \ignorespaces%
}

\def\joinQuaternary{
    \setbox\myBoxA=\hbox{\theHypSeparation}%
    \lcurScoreEnd=\rrcurScoreEnd%
    \advance\lcurScoreEnd by\wd\rcurBox%
    \advance\lcurScoreEnd by\wd\lcurBox%
    \advance\lcurScoreEnd by\wd\ccurBox%
    \advance\lcurScoreEnd by3\wd\myBoxA%
    \displace=\lcurScoreEnd%
    \advance\displace by -\lcurScoreStart%
    \lcurCenter=.5\displace%
    \advance\lcurCenter by\lcurScoreStart%
    \ifx\rootAtBottomFlag\myTrue%
        \setbox\lcurBox=%
            \hbox{\box\lcurBox\unhcopy\myBoxA\box\ccurBox%
                      \unhcopy\myBoxA\box\rcurBox
                      \unhcopy\myBoxA\box\rrcurBox}%
    \else%
        \htLbox = \ht\lcurBox%
        \htAbox = \ht\myBoxA%
        \htCbox = \ht\ccurBox%
        \htRbox = \ht\rcurBox%
        \htRRbox = \ht\rrcurBox%
        \setbox\lcurBox=%
            \hbox{\lower\htLbox\box\lcurBox%
                  \lower\htAbox\copy\myBoxA\lower\htCbox\box\ccurBox%
                  \lower\htAbox\copy\myBoxA\lower\htRbox\box\rcurBox%
                  \lower\htAbox\copy\myBoxA\lower\htRRbox\box\rrcurBox}%
    \fi%
    \displace=\newCenter%
    \advance\displace by -.5\newScoreStart%
    \advance\displace by -.5\newScoreEnd%
    \advance\lcurCenter by \displace%
    \edef\curBox{\lcurBox}%
    \edef\curScoreStart{\lcurScoreStart}%
    \edef\curScoreEnd{\lcurScoreEnd}%
    \edef\curCenter{\lcurCenter}%
    \joinUnary%
}

\usepackage{tikz}
\usetikzlibrary{decorations.pathmorphing,shapes}

\newcounter{sarrow}

\renewcommand{\leq}{\leqslant}

\def\HW{\mbox{\tiny HW}}

\makeatletter
\def\subsection{\@startsection{subsection}{3}%
  \z@{.5\linespacing\@plus.7\linespacing}{.3\linespacing}%
  {\bfseries\centering}}
\makeatother

\makeatletter
\def\subsubsection{\@startsection{subsubsection}{3}%
  \z@{.5\linespacing\@plus.7\linespacing}{.3\linespacing}%
  {\centering}}
\makeatother

\makeatletter
\def\myfnt{\ifx\protect\@typeset@protect\expandafter\footnote\else\expandafter\@gobble\fi}
\makeatother

\theoremstyle{definition}

\newtheorem{theorem}{Theorem}[section]
\newtheorem{definition}[theorem]{Definition}
\newtheorem{lemma}[theorem]{Lemma}

\newtheorem{example}[theorem]{Example}
\newtheorem{corollary}[theorem]{Corollary}

\newtheorem{fact}[theorem]{Fact}

\newcounter{claimcounter}
\numberwithin{claimcounter}{theorem}

\def\presuper#1#2%
  {\mathop{}%
   \mathopen{\vphantom{#2}}^{#1}%
   \kern-\scriptspace%
   #2}

\begin{document}

\title{A Logic for Arguing About Probabilities in Measure Teams}
\thanks{The research of the second author was supported by the Finnish Academy of Science and Letters (Vilho, Yrj\"o and Kalle V\"ais\"al\"a foundation).}

\author{Tapani Hyttinen}
\address{Department of Mathematics and Statistics,  University of Helsinki, Finland} 

\author{Gianluca Paolini}
\address{Department of Mathematics and Statistics,  University of Helsinki, Finland} 

\author{Jouko V\"a\"an\"anen}
\address{Department of Mathematics and Statistics,  University of Helsinki, Finland and Institute for logic, Language and Computation, University of Amsterdam, The Netherlands} 

\begin{abstract} We use sets of assignments, a.k.a. teams, and measures on them to define probabilities of first-order formulas in given data. We then axiomatise first-order properties of such probabilities and prove a completeness theorem for our axiomatisation.  We use the Hardy-Weinberg Principle of biology and the Bell's Inequalities of quantum physics as  examples. 
\end{abstract}

\maketitle

\def\ma{\mathcal A}

\section{Introduction}

The logic of propositions with assigned probabilities is usually associated with nondeductive methods  such as inductive reasoning (\cite{MR0040253}). The concept of probability in such an approach is the degree of confirmation or belief. Instead, in this paper we assign probabilities to propositions using  the {\em frequency} interpretation and study properties of such probabilities. Thus, while probability logic usually focuses on the question how to assign probabilities to composite formulas, we focus on the symmetric question how to axiomatise formulas built up from probabilities. To make using the frequency interpretation possible in  defining probabilities we adopt the approach of {\em team semantics} from \cite{vaananen}.  
 
Suppose $\ma$ is a first-order structure with domain $A$. Suppose furthermore $v_0,\ldots,v_n$ are variables that have values in $A$. If we have a set $X$ of assignments of values to $v_0,\ldots,v_n$ in $A$, called a {\em team}, we may ask, what is the probability that a randomly chosen assignment in $X$ satisfies a given first-order formula $\phi(v_0,\ldots,v_n)$ in $\ma$? For this to make perfect sense we need to specify a probability function for relevant subsets of $X$. Our {\em measure teams}
are exactly such teams. In this paper we give axioms for making inferences about first-order properties of such probabilities, and prove the completeness of our axioms.

In the context of experimental science it is natural to consider probabilities of formulas rather than just the truth values true/false. In the world of Big Data this is even more relevant. We suggest  to take the concept of a measure team as a starting point and use it to compute the probabilities of formulas, rather than having the probabilities as given, as in the probability logic of \cite{MR0040253, MR0175755}. In a sense we can argue about the probabilities and have the evidence---the data, or team as we call it---as part of the discussion.

The measure teams that arise from actual experiments are, of course, finite. Indeed, the simplest measure teams consist just of a finite number of assignments of values to fixed variables, as in the table Figure~\ref{dis} of 8 rows of binary data.
\begin{figure}[h]
$$\begin{array}{|c|c|c|c|c|}
\hline
\phantom{a} & v_0 & v_1 & v_2 & v_3 \\
\hline
	0 & 1 & 1 & 0 & 1 \\
	1 & 1 & 1 & 1 & 1 \\
	2 & 1 & 1 & 1 & 1 \\
	3 & 1 & 1 & 1 & 0 \\
	4 & 0 & 0 & 1 & 1 \\
	5 & 0 & 0 & 0 & 0 \\
	6 & 0 & 0 & 0 & 0 \\
	7 & 0 & 0 & 0 & 0 \\	
\hline
\end{array}
$$\caption{A discrete measure team  \label{dis}}\end{figure}	
An example of a finite measure team in biology is a pool of genes. One of the pioneering mathematical results in genetics is the Hardy-Weinberg Theorem which shows that a conservation phenomenon takes place in a gene pool from generation to generation under certain assumptions, such as random mating. The Hardy-Weinberg Theorem is an example
of a property of measure teams that can be expressed and proved in our setup.

Despite the finiteness of teams arising from experiments, we consider in this paper mainly infinite teams, typically continuum size, which abstract away the finiteness of empirical observations. Our Completeness Theorem (Theorem~\ref{main_theorem}) is with respect to infinite measure teams. A paradigm example is an idealised measurement of given variables $v_0,\ldots,v_n$ at all points of time starting at time $0$ and ending at time $1$ (see Figure~\ref{flow}). The values of the variables can be e.g. real numbers which change continuously with time. Thus we have an assignment $s_t$ that depends continuously on time $t$ and interprets variables $v_0,\ldots,v_3$ at every point of time. When time progresses from $0$ to $1$, the vector $(s_t(v_0),\ldots,s_t(v_3))$ flows from $(s_0(v_0),\ldots,s_0(v_3))$ to $(s_1(v_0),\ldots,s_1(v_3))$. It seems appropriate to call such teams {\em continuous teams} as the assignment changes continuously with time. In physical sciences   variables, such as temperature, speed, pressure, amplitude, force, etc, are typically continuous in time. Therefore the concept of continuous team would seem to cover a lot of examples. Continuous teams (when considered with respect to the Lebesgue measure) are examples of measure teams, the topic of this paper.

	\begin{figure}[h]
$$\begin{array}{|c|c|c|c|c|}
\hline
t & v_0 & v_1 & v_2 & v_3 \\
\hline
0	 &s_0(v_0) & s_0(v_1) & s_0(v_2) & s_0(v_3) \\
\vdots	 &\vdots & \vdots & \vdots & \vdots \\
t	 &s_t(v_0) & s_t(v_1) & s_t(v_2) & s_t(v_3) \\
\vdots	 &\vdots & \vdots & \vdots & \vdots \\
1	 &s_1(v_0) & s_1(v_1) & s_1(v_2) & s_1(v_3) \\
\hline
\end{array}
$$\caption{A continuous measure team  \label{flow}}\end{figure}


\section{Measure Teams}	 

	We denote by $\mathrm{Var} = \left\{ v_i \, | \, i < \omega \right\}$ the set of individual first-order variables. 
	
	
	\begin{definition}[Multi-team] A {\em multi-team} $X$ with values in $A$ and domain $\mathrm{dom}(X) \subseteq \mathrm{Var}$ is a pair $(\Omega, \tau)$ such that $\Omega$ is a set and $\tau: \Omega \rightarrow A^{\mathrm{dom}(X)}$ is a function.
\end{definition}

	Given a multi-team $(\Omega, \tau)$, if we put a probability measure on $\Omega$ we get a probabilistic notion of team. Of course, for this definition to be useful one has to put also some measurability conditions on $\tau$. This idea leads to the notion of {\em measure team}, which is the focus of the present paper.

	\begin{definition}[Measure team] Let $L$ be a signature and $\mathcal{A}$ an $L$-structure. A measure team $X$ with values in $\mathcal{A}$ and domain $\mathrm{dom}(X) \subseteq \mathrm{Var}$ is a quadruple $(\Omega, \mathcal{F}, P, \tau)$ such that $(\Omega, \mathcal{F}, P)$ is a probability space and $\tau: \Omega \rightarrow A^{\mathrm{dom}(X)}$ is a measurable function, in the sense that 
	$$\left\{ i \in \Omega \, | \, \mathcal{A} \models_{t(i)} \phi \right\} \in \mathcal{F}$$
for every first-order $L$-formula $\phi$ with free variables in $\mathrm{dom}(X)$.
\end{definition}

	If $X =(\Omega, \mathcal{F}, P, \tau)$ is countable, then the natural choice for $\mathcal{F}$ is $\mathcal{P}(\Omega)$, i.e. the whole power set of $\Omega$, and measurability of $\tau$ is automatically ensured. In the uncountable case, the situation is of course more delicate.

\begin{definition}[Probability]\label{def_prob} Let $L$ be a signature, $\mathcal{A}$ an $L$-structure, $X = (\Omega, \mathcal{F}, P, \tau)$ a measure team with values in $\mathcal{A}$ and $\phi$ a first-order $L$-formula with free variables in $\mathrm{dom}(X)$. We let
			\[ [\phi]_X = P(\left\{ i \in \Omega \, | \, \mathcal{A} \models_{\tau(i)} \phi \right\}) .\]	
\end{definition}

	That is, $[\phi]_X$ is the probability that a randomly chosen assignment from $X$ satisfies $\phi$. Notice that because of the measurability conditions imposed on $\tau$, the above definition makes sense.	
	
\begin{example}\label{boolean}	In Figure \ref{dis} we have an example of a measure team $X = (\Omega, \mathcal{F}, P, \tau)$ with values in the boolean algebra on two elements $\mathcal{A} = (\{0,1\}, 0, \vee, \wedge, \neg)$, where $(\Omega, \mathcal{P}(\Omega), P)$ is the set with eight elements endowed with the normalized counting measure (measure of one point is $\frac{1}{8}$), the domain of $X$ is $\left\{ v_0, v_1, v_2, v_3 \right\}$ and $\tau$ is as in the figure, e.g. $\tau(0)((v_0, v_1, v_2, v_3)) = (1, 1, 0, 1)$. If we consider the variables $v_i$ as propositional variables, then in this case 
	$$ [v_0 \wedge v_1]_X = \frac{1}{2}, $$
because 50\% of the rows satisfy the propositional formula $v_0\wedge v_1$.
We will call measure teams of this particular kind boolean multi-teams. In  \cite{QTL} a system of propositional logic based on boolean multi-teams has been investigated.
\end{example}	

	We denote by $\mathcal{R} = (\mathbb{R}, 0, 1, +, -, \cdot, \leq)$ the ordered field of real numbers, with $\mathcal{L}$ the $\sigma$-algebra of Lebesgue measurable subsets of $[0, 1]$ and with $P$ the Lebesgue measure on $[0, 1]$.

\begin{example} Let $(f_i: [0, 1] \rightarrow \mathbb{R}))_{i < 3}$ be continuous functions, and let $\tau: [0, 1] \rightarrow \mathbb{R}^{\left\{ v_0, v_1, v_2 \right\}}$ so that $\tau(a)(v_i) = f_i(a)$, for $i < 3$. Then $X = ([0, 1], \mathcal{L}, P, \tau)$ is a measure team, which we called above {\em continuous measure team}, with values in $\mathcal{R}$. This follows from elementary properties of continuous functions and elimination of quantifiers for $\mathcal{R}$.
\end{example}


\section{Measure Team Logic}\label{measure_team_logic}

Our measure team logic is concerned with making inferences about the probabilities themselves, not about how probabilities of composite formulas depend on probabilities of the subformulas. 
An example of a valid sentence of our measure team logic is
$$|\phi| = |\phi \wedge \psi| + |\phi \wedge \neg \psi|,$$ where $|\phi|$ denotes the probability of $\phi$ in the team in question. Thus our logic has built in function symbols $+,\cdot$ for expressing arithmetic relations between probabilities.


	We now define {\em measure team logic}. Let $L_0$ be a countable signature, $Q \subseteq \mathbb{R}$ countable and $n \leq \omega$. The intended $Q$ is the set of rational numbers $\mathbb{Q}$ (or even $\mathbb{Q} \cap [0, 1]$). We define the signature $L_Q$ and $L_1$ as follows
	\[ L_Q = \left\{ 0, 1, +, -, \cdot, \leq \right\} \cup \left\{ c_q \, | \, q \in Q \right\}\]
	\[ L_1 = L_1(L_0, n) = L_Q \cup \left\{ |\phi(x)| \; | \; \phi(x) \; \text{$L_0$-formula } x = (v_i)_{i < n} \right\},\]
where the $c_q$ and $|\phi(x)|$ are constant symbols. Note that $|\phi(x)|$ is considered just as a constant symbol, however complicated the formula $\phi(x)$ is. Without loss of generality we may assume that $L_0 \cap L_1 = \emptyset$, this is to avoid possible confusion.

A typical (atomic) formula of our logic is of the form
$$|\phi(x)|=|\psi(x)|$$
with the meaning that a randomly chosen assignment from our team is as likely to satisfy $\phi(x)$ as it is to satisfy $\psi(x)$.
Another typical (atomic) formula  is of the form
$$|\phi(x)|=|\psi(x)|+|\theta(x)|$$
with the meaning that the probability that a randomly chosen assignment from our team satisfies $\phi(x)$ is the sum of the corresponding probabilities for  $\psi(x)$ and $\theta(x)$.

	Given a measure team $X$ with values in $\mathcal{A}$ and $\mathrm{dom}(X) = \left\{ v_i \, | \, i < n \right\}$, we let $\mathcal{R}^{X}_Q$ be the expansion of $\mathcal{R} = (\mathbb{R}, 0, 1, +, -, \cdot, \leq)$ to an $L_1$-structure obtained by interpreting the constant $c_q$ as the real number $q$, and by letting 
	$$|\phi(x)|^{\mathcal{R}^{X}_Q} = [\phi(x)]_X,$$ 
where $[\phi(x)]_X$ is as in Definition \ref{def_prob}. Thus, $[\phi(x)]_X$ is the value of the constant symbol $|\phi(x)|$ in $\mathcal{R}^{X}_Q$.

	\begin{definition}[Semantics] Let $\Sigma$ be an $L_1$-theory, $\mathcal{A}$ an $L_0$-structure, $X$ a measure team with values in $\mathcal{A}$ and $\mathrm{dom}(X) = \left\{ v_i \, | \, i < n \right\}$. We define 
	$$X \models \Sigma \;\; \Leftrightarrow_{\tiny def} \;\; \mathcal{R}^{X}_Q \models \Sigma.$$
\end{definition}

	\begin{definition}[Logical consequence] Let $T$ be an $L_0$-theory and $\Sigma \cup \left\{ \alpha \right\}$ an $L_1$-theory. We define $$(T, \Sigma) \models \alpha$$ if for every $\mathcal{A} \models T$ and every measure team $X$ with values in $\mathcal{A}$ such that $\mathrm{dom}(X) = \left\{ v_i \, | \, i < n \right\}$, we have that 
	$$X \models \Sigma \;\; \Rightarrow \;\; X \models \alpha.$$
\end{definition}

	We now define a deductive system $(T, \Sigma) \vdash \alpha$ with $T$ an $L_0$-theory, $\Sigma$ an $L_1$-theory and $\alpha$ an $L_0$-formula or an $L_1$-formula. Of course what we are really interested in is the case when $\alpha$ is an $L_1$-formula, but for things to work, i.e. to prove completeness, we also have to admit the case in which $\alpha$ is an $L_0$-formula. The deductive system $\vdash$ has three components: $\vdash_0$, $\vdash_1$ and 	
$\vdash_2$. The component $\vdash_0$ allows to deduce $L_0$-formulas from $L_0$-formulas, the component $\vdash_1$ allows to deduce $L_1$-formulas from $L_1$-formulas and the component $\vdash_2$ allows to deduce $L_1$-formulas from $L_0$-formulas. The component $\vdash_0$ is simply the deductive system of first-order logic with respect to $L_0$-formulas. The component $\vdash_1$ is the deductive system of first-order logic with respect to $L_1$-formulas plus the axioms $\mathrm{RCF}^* = \mathrm{Th}(\mathcal{R}_Q)$ (or any axiomatization thereof). Finally, the component $\vdash_2$ consists of three axioms ($A_0$)-($A_2$) and one rule ($R_0$), as below:

\smallskip

\begin{prooftree}
\AxiomC{}
\LeftLabel{($A_0$)}
\UnaryInfC{$|\phi \wedge \neg \phi| = 0$}
\end{prooftree}

\begin{prooftree}
\AxiomC{}
\LeftLabel{($A_1$)}
\UnaryInfC{$|\phi \vee \neg \phi| = 1$}
\end{prooftree}

\begin{prooftree}
\AxiomC{}
\LeftLabel{($A_2$)}
\UnaryInfC{$|\phi \vee \psi| = |\phi| + |\psi| - |\phi \wedge \psi|$}
\end{prooftree}
	
\begin{prooftree}
\AxiomC{$\bigvee_{i < k} \bigwedge_{j < m_i} \forall x (\phi^i_j(x) \rightarrow \psi^i_j(x))$}
\LeftLabel{($R_0$)}
\UnaryInfC{$\bigvee_{i < k} \bigwedge_{j < m_i} (|\phi^i_j(x)| \leq |\psi^i_j(x)|)$}
\end{prooftree}
where in rule ($R_0$) we assume that the formulas $\forall x (\phi^i_j(x) \rightarrow \psi^i_j(x))$ are sentences.

\medskip	

	As an example of the use of our deductive system we show that $$\vdash |\phi| = |\phi \wedge \psi| + |\phi \wedge \neg \psi|.$$ 
	First of all notice that 
		\begin{prooftree}
\AxiomC{$\phi \wedge \psi \wedge \phi \wedge \neg\psi \leftrightarrow \psi \wedge \neg \psi$}
\LeftLabel{($R_0$)}
\UnaryInfC{$|\phi \wedge \psi \wedge \phi \wedge \neg\psi| = |\psi \wedge \neg \psi|$}
\AxiomC{}
\LeftLabel{($A_0$)}
\UnaryInfC{$|\psi \wedge \neg \psi| = 0$}
\BinaryInfC{$|\phi \wedge \psi \wedge \phi \wedge \neg\psi| = 0$}
\end{prooftree}
Secondly, let $\alpha = |(\phi \wedge \psi) \vee (\phi \wedge \neg\psi)| = |\phi \wedge \psi| + |\phi \wedge \neg\psi| - |\phi \wedge \psi \wedge \phi \wedge \neg\psi|$ and notice that
\begin{prooftree}
\AxiomC{}
\LeftLabel{($A_2$)}
\UnaryInfC{$\alpha$}
\AxiomC{$|\phi \wedge \psi \wedge \phi \wedge \neg\psi| = 0$}
\BinaryInfC{$|(\phi \wedge \psi) \vee (\phi \wedge \neg\psi)| = |\phi \wedge \psi| + |\phi \wedge \neg\psi|$}
\end{prooftree}
Finally, we conclude
	\begin{prooftree}
\AxiomC{$\phi \leftrightarrow (\phi \wedge \psi) \vee (\phi \wedge \neg\psi)$}
\LeftLabel{($R_0$)}
\UnaryInfC{$|\phi| = |(\phi \wedge \psi) \vee (\phi \wedge \neg\psi)|$}
\AxiomC{$|(\phi \wedge \psi) \vee (\phi \wedge \neg\psi)| = |\phi \wedge \psi| + |\phi \wedge \neg\psi|$}
\BinaryInfC{$|\phi| = |\phi \wedge \psi| + |\phi \wedge \neg \psi|$}
\end{prooftree}
\medskip

\section{Some examples}

\begin{example}
In this example we look at the usefulness
of quantification when one expresses conditions on probabilities.
%
This example is hypothetical in many senses
but it is faithful to
the calculations of quantum mechanics.

Suppose that we have two observables $v_{1}$ and $v_{2}$
which can take values from the set $\{ 1,2,3,4\}$,
a device that produces particles such that they are all
in the same unknown pure state and that someone has
produced a large table $X$ of measurements of these observables
from the particles produced by the device
(usually it is impossible to measure 
the two observables independently 
from one particle
but we overlook this kind of problems here,
in \cite{QTL} we have studied logical questions related to
the impossibility of experimentally producing tables 
with values for all observables from every particle).

In physics this kind of situation is typically modelled by two
self-adjoint operators in a 4-dimensional Hilbert space.
Let $P$ be the operator for $v_{1}$ and $p(i)$, $i\in\{ 1,2,3,4\}$,
its eigenvectors with eigenvalue $i$. Similarly,
let $Q$ be the operator for $v_{2}$ and  $q(i)$ its eigenvectors.
Notice that when one knows the operators
$P$ and $Q$, it is possible to calculate the coordinates of the vectors
$p(i)$ in the basis of eigenvectors of $Q$.

Can we express in measure team logic 
the condition that the measurements are in harmony with
the theory? Yes, the following is expressible
in our logic: 
there are four complex numbers (pairs of reals)
$c_{n}$, $n\in\{ 1,2,3,4\}$, such that for all
$i\in\{ 1,2,3,4\}$ the following holds:
$\vert c_{i}\vert^{2}=[v_{1}=i]_{X}$ and
$\vert \langle s\vert q(i)\rangle\vert^{2}=[v_{2}=i]_{X}$, where
$\langle \cdot\vert \cdot\rangle $ is the inner product and 
$$s=(1/2)\sum_{n=1}^{4}c_{n}p(n).$$
This is exactly the condition that our data $X$ agrees with the theory.

\end{example}

\begin{example}
This is an example of the use of $T$
in theories $(T,\Sigma )$. We look at homogeneous Markov
chains (see e.g. \cite[pg. 61]{markov_chains}). We think of variables $(v_i)_{i<\omega}$ as random variables and elements of the team $X$ as tests. The value of the random variable $v_j$ in the test $i\in\Omega$ is  $\tau(i)(v_j)$. Figuratively speaking, the team $X$ consists of rows of data concerning the random variables $v_i$.  We give axioms which say that the sequence $(v_i)_{i<\omega}$ of random variables is a Markov process. 
The state space of a Markov process is usually assumed to be
countable which is not a first-order property. However, if one looks at
chains in which the state diagram has some additional properties
(after we remove some of the 
arrows with probability $0$) we can
overcome this problem. The additional property we study here
is that there is some natural number $N$ such that
from each node in the state diagram at most $N$ arrows with non-zero
probability go out. We also assume that the chain has an
initial state from which every process starts.
The main example in our mind of this
is the random walk in a space of dimension $N/2$.

\def\n{\eta}

Markov chains with these properties can be axiomatized as follows
in measure team logic:
The vocabulary of $T$ consists of a binary relation $E$
and constants $c_{\n}$, $\n\in N^{<\omega}$.
The theory $T$ says the following for all
$\n,\xi\in N^{<\omega}$:

(a) $(c_{\n},a)\in E$ iff $a=c_{\n\frown (i)}$ for some $i<N$.

(b) If $c_{\n}=c_{\xi}$ then for all $i<N$,
$c_{\n\frown (i)}=c_{\xi\frown (i)}$.

\noindent
As an initial state we take
(the interpretation of) $c_{\empty}$ and notice that
if $\mathcal{A}\models T$, then the set $G_{\mathcal{A}}$ of
the interpretations of the constants
equipped with $E\cap (G_{\mathcal{A}}\times G_{\mathcal{A}})$ is a
connected directed graph
and every state diagram satisfying our assumptions
(after removing some of the useless arrows) can be obtained from
a model of $T$ in this way.
Also it is worth noticing that
any process that starts from the initial state stays inside $G_{\mathcal{A}}$.

We let $n=\omega$ and describe the probabilities as a Markov process
that starts from the initial state. Thus $\Sigma$ consists of
the following for all $i,j<\omega$, $\n\in N^{<\omega}$ and $k<N$:

(A) $\vert v_{0}=c_{\empty}\vert =1$.

(B) $\vert E(v_{i},v_{i+1})\vert =1$.

(C) $(\vert v_{i}=c_{\n}\vert =0)\vee (\vert v_{i}=c_{\n}\vert =0)
\vee 
\newline \phantom{C} \phantom{C} \phantom{C}\phantom{a\,} 
(\vert v_{i}=c_{\n}\wedge v_{i+1}=
c_{\n\frown (k)}\vert\vert v_{j}=c_{\n}\vert =
\vert v_{j}=c_{\n}\wedge v_{j+1}=
c_{\n\frown (k)}\vert\vert v_{i}=c_{\n}\vert )$.

A team $X$ satisfies $\Sigma$ if and only if the stochastic process consisting of the values of the random variables $(v_i)_{i<\omega}$ in $X$ is a Markov chain.
\end{example}

\begin{example}[The Hardy-Weinberg Principle]

In the early days of biology there was an apparent paradox: It seemed that in any population the dominant alleles should eventually drive out the recessive ones, but this was not supported by observations and experimental data. 
The Hardy-Weinberg Principle (\cite{hardy,weinberg}) explains why in a randomly mating 
population the recessive alleles stabilise to maintain a fixed portion, even after just one generation. 

We consider a diallelic gene with alleles A and a. The logically---but not at all biologically---possible genotypes form the 27 element set
$$M=\{AA,Aa,aa\}\times\{AA,Aa,aa\}\times\{AA,Aa,aa\},$$
where the first component of the triples is the genotype of the father, the second that of the mother and the third that of the child.
%
%
%
%
%
%
\def\mm{\mathcal{M}}
Let $L_0$ be the following signature ($f$ for father, $m$ for mother and $c$ for child)
	\[ \left\{ P^j_k \, | \, j \in \left\{ f, m, c \right\} \text{ and } k \in \left\{ AA, Aa, aa \right\} \right\},\]
where the $P^j_k$ are unary predicate symbols. We get an $L_0$-structure by defining
\[\mm=(M, (P^j_k)^\mm)_{j,k},\]
where 
$$(P^f_k)^\mm=\{(u,v,w) \in M: u=k \},$$
$$(P^m_k)^\mm=\{(u,v,w) \in M: v=k \} \text{ and } (P^c_k)^\mm=\{(u,v,w) \in M: w=k \}.$$
Thus $\mm$ is simply the set $M$ of logically possible genotypes with their internal structure accessible via the predicates $P^j_k$. Now we can look at measure teams of assignments of variables in this structure. We focus on three variables $v_0,v_1$ and $v_2$, representing three generations. Such a measure team can be thought of as genetic data about three generations of a population. We disregard mating across generations, so for us the next generation is always the children. 

More formally, a measure team $X$, relevant for the purpose of the Hardy-Weinberg Principle, is a quadruple $(\{1,\ldots, n\}, \mathcal{P}(\{1,\ldots, n\}), P, \tau)$ such that $P$ is the uniform probability on  $\{1,\ldots, n\}$ and $\tau: \{1,\ldots, n\} \rightarrow M^{\{v_0,v_1,v_2\}}$ is an arbitrary function. For  each $i\in\{1,\ldots,n\}$ the assignment $\tau(i)$ records a father-mother-child triple of the first generation ($v_0$), second generation ($v_1$) and the third generation ($v_2$).


For the language $L_1$ we choose $Q=\mathbb{Q}$. Let $\Sigma_{\HW}$ consist of the below  $L_1$-equations  (\ref{starr})-(\ref{starstarstarstarstar}). Remember that the language $L_1$ contains all the constant symbols $|\phi(x)|$, where $\phi(x)$ is an arbitrary $L_0$-formula. So (\ref{starr})-(\ref{starstarstarstarstar}) are atomic sentences, more exactly equations of constant terms.
	\begin{equation}\label{starr}
	|P^j_k(v_{i+1})| =  |P^c_k(v_i)|, 
\end{equation} 
for $j = f, m$, $k = AA, Aa, aa$ and $i =0, 1$;
  \begin{equation}\label{starstar}
	|P^f_{k}(v_{i+1}) \wedge P^m_{z}(v_{i+1})| = |P^f_{k}(v_{i+1}) \wedge P^m_{z}(v_{i+1}) \wedge P^c_{w}(v_{i+1})|
\end{equation} 
for $(k, z, w) = (AA, AA, AA), (AA, aa, Aa), (aa, AA, Aa), (aa, aa, aa)$ and $i =0, 1$;
  \begin{equation}\label{starstarstar}
	|P^f_{k}(v_{i+1}) \wedge P^m_{l}(v_{i+1})| = 2\cdot |P^f_{k}(v_{i+1}) \wedge P^m_{l}(v_{i+1}) \wedge P^c_{m}(v_{i+1})|
\end{equation} 
for $(k, l, m) = (AA, Aa, AA), (AA, Aa, Aa), (aa, Aa, Aa), (aa, Aa, aa),(Aa, aa, aa),$ $(Aa, AA, Aa), (Aa, aa, Aa), (Aa, AA, AA), (Aa, Aa, Aa)$ and $i =0, 1$;
  \begin{equation}\label{starstarstarstar}
	|P^f_{k}(v_{i+1}) \wedge P^m_{l}(v_{i+1})| = 4\cdot |P^f_{k}(v_{i+1}) \wedge P^m_{l}(v_{i+1}) \wedge P^c_{m}(v_{i+1})|
\end{equation} 
for $(k, l, m) = (Aa, Aa, AA), (Aa, Aa, aa)$ and $i =0, 1$;
	\begin{equation}\label{starstarstarstarstar}
	|P^f_k(v_{i+1}) \wedge P^m_l(v_{i+1})| =  |P^f_k(v_{i+1})| \cdot |P^m_l(v_{i+1})|
\end{equation} 
for $k = AA, Aa, aa$, $l = AA, Aa, aa$ and $i =0, 1$.

The formulas of type (1) express that allele frequencies are equal in the sexes, the formulas of type (2)-(4) specify how the genotypes are inherited, according to Mendel's Principles, and the formulas of type (5) express that mating is random, an important assumption of the Hardy-Weinberg Principle.

Finally, let $\alpha_{\HW}$ be the conjunction of the following $L_1$-equations: 
$$ 
\begin{array}{lcr}
|P^c_{AA}(v_1)| &=& |P^c_{AA}(v_2)|\\ 
|P^c_{Aa}(v_1)| &=& |P^c_{Aa}(v_2)|\\
|P^c_{aa}(v_1)| &=& |P^c_{aa}(v_2)|.
\end{array}$$
These conjuncts say that the genotype frequencies among the children in the second and third generations are the same, i.e. a stable balance achieves already at the second generation.

Thus, all the assumptions of the Hardy-Weinberg Principle are formalizable in our logic. Since the Hardy-Weinberg Principle is true and our logic is complete (see Section \ref{completeness}), it follows that
$$\Sigma_{\HW}\vdash\alpha_{\HW},$$
i.e. our deductive system proves (a formalization) of the Hardy-Weinber Principle.
\end{example}

\begin{example}[Bell's Inequalities]

	 In \cite{QTL}, among other things, we presented a system of probability logic capable to handle so-called {\em logical Bell's inequalities} \cite{abramsky}. 	 
	 Suppose $X=(\Omega, \mathcal{F}, P, \tau)$ is a boolean multi-team (see Example \ref{boolean}) the domain of which contains the proposition symbols of some given propositional formulas
$(\phi_j)_{j < k}$. Then
	
\begin{equation}\label{star}
 \sum_{j < k} [\phi_j]_X \leq k-1 + [\bigwedge_{j < k} \phi_j]_X. 
  \end{equation}
If furthermore  the formula $\bigwedge_{j < k} \phi_j$ is contradictory (in the sense of propositional logic), then $[\bigwedge_{j < k} \phi_j]_X =0$. Thus, the inequality (\ref{star}) becomes
	\begin{equation}\label{starstar} \sum_{j < k} [\phi_j]_X \leq k-1. \end{equation}
	Inequalities of this form (\ref{starstar}) are of great importance in foundations of quantum mechanics, see \cite{abramsky} and \cite{QTL}. Because of the completeness result presented in the next section, we will see that this inequalities are provable in our logic. For suitably chosen propositional formulas
$(\phi_j)_{j < k}$, representing propositions about Quantum Mechanics, the inequality (\ref{starstar}) fails thereby demonstrating the contextuality of probabilities in the quantum world. To remedy this  a {\em quantum team logic} is introduced in \cite{QTL}. In the quantum team logic the problematic inequalities (\ref{starstar}) are not provable but we still have a Completeness Theorem with respect to {\em quantum teams}, a generalization of the concept of a measure team.
\end{example}
	
\section{Completeness}\label{completeness}

	In this section we prove that the deductive system described in Section \ref{measure_team_logic} is complete with respect to the given semantics. We begin with a Lindenbaum's Lemma like result for our deductive system. As in Section \ref{measure_team_logic}, let $L_0$ be a countable signature, $Q \subseteq \mathbb{R}$ countable and $n \leq \omega$.
	
	\begin{lemma}\label{linden} Suppose that $(T, \Sigma) \nvdash \perp$. Then there are a complete $L_0$-theory $T_0$ and a complete $L_1$-theory $\Sigma_0$ such that $T \subseteq T_0$, $\Sigma \subseteq \Sigma_0$ and $(T_0, \Sigma_0) \nvdash \perp$.
\end{lemma}

	\begin{proof} We first construct $T_0$ as a limit of a chain $(T^*_i)_{i < \omega}$ of $L_0$-theories. Let $(\phi_i)_{i < \omega}$ be an enumeration of the $L_0$-sentences. By induction on $i < \omega$ we construct $T^*_i$ so that $(T^*_i, \Sigma) \nvdash \perp$ and either $\phi_i \in T^*_{i+1}$ or $\neg \phi_i \in T^*_{i+1}$. If $i = 0$, let $T^*_0 = T$. If $i = j+1$, there are three cases.
\newline {\bf Case 1}. $T^*_j \vdash \phi_j$. Let $T^*_i = T^*_j \cup \left\{ \phi_j \right\}$.
\newline {\bf Case 2}. $T^*_j \vdash \neg \phi_j$. Let $T^*_i = T^*_j \cup \left\{ \neg \phi_j \right\}$.
\newline {\bf Case 3}. $T^*_j \nvdash \phi_j$, i.e. $T^*_j \cup \left\{ \neg \phi_j \right\} \nvdash \perp$, and $T^*_j \nvdash \neg \phi_j$, i.e. $T^*_j \cup \left\{ \phi_j \right\} \nvdash \perp$. For the sake of a contradiction, suppose that 
	\[(T^*_j \cup \left\{ \phi_j \right\}, \Sigma) \vdash \perp \; \text{ and } \; (T^*_j \cup \left\{ \neg \phi_j \right\}, \Sigma) \vdash \perp.\]
We show this is impossible and then extend $T^*_j$ with $\phi_j$ if $(T^*_j \cup \left\{ \phi_j \right\}, \Sigma) \nvdash \perp$, and $\neg \phi_j$ otherwise. Given that $T^*_j \cup \left\{ \phi_j \right\} \nvdash \perp$, there must exists $t < \omega$ so that letting
$$\chi_s = \bigvee_{i < k_s} \bigwedge_{j < m_{(i, s)}} \forall x(\phi^s_{(i,j)} \rightarrow \psi^s_{(i,j)}) \; \text{ and } \; \chi'_s = \bigvee_{i < k_s} \bigwedge_{j < m_{(i, s)}} (|\phi^s_{(i,j)}| \leq |\psi^s_{(i,j)}|),$$
for $s \leq t$, we have that
\def\fcenter{\mbox{$\cdots$}}
\begin{prooftree}
\AxiomC{$T^*_j \cup \left\{ \phi_j \right\}$}
\UnaryInfC{$\chi_0$}
\LeftLabel{($R_0$)}
\UnaryInfC{$\chi'_0$}
\AxiomC{$T^*_j \cup \left\{ \phi_j \right\}$}
\UnaryInfC{$\phantom{\chi_0} \cdots \phantom{\chi_0}$}
\LeftLabel{($R_0$)}
\UnaryInfC{$\phantom{\chi'_0} \cdots \phantom{\chi_0}$}
\AxiomC{$T^*_j \cup \left\{ \phi_j \right\}$}
\UnaryInfC{$\chi_t$}
\LeftLabel{($R_0$)}
\UnaryInfC{$\chi'_t$}
\AxiomC{$\Sigma$}
\QuaternaryInfC{$\perp$}	
\end{prooftree}
Notice though that our deductive system proves that formulas in $\bigwedge \bigvee \bigwedge$ form are equivalent to formulas in $\bigvee \bigwedge$ form, and so we have that

\def\fcenter{\mbox{$\cdots$}}
\begin{prooftree}
\AxiomC{$\chi_0 \, \fcenter \, \chi_t$}
\UnaryInfC{$\bigwedge_{s \leq t} \chi_s$}
\UnaryInfC{$\chi$}
\UnaryInfC{$\bigwedge_{s \leq t} \chi_s$}
\UnaryInfC{$\chi_0 \, \fcenter \, \chi_t$}	
\end{prooftree}
where $\chi$ is the formula in $\bigvee \bigwedge$ form equivalent to $\bigwedge_{s \leq t} \chi_s$. In substance, without loss of generality we can simplify the situation assuming that $t=0$ and thus
\begin{prooftree}
\AxiomC{$T^*_j \cup \left\{ \phi_j \right\}$}
\UnaryInfC{$\bigvee_{i < k_0} \bigwedge_{j < m_{(i, 0)}} \forall x(\phi^0_{(i,j)} \rightarrow \psi^0_{(i,j)})$}
\LeftLabel{($R_0$)}
\UnaryInfC{$\bigvee_{i < k_0} \bigwedge_{j < m_{(i, 0)}} (|\phi^0_{(i,j)}| \leq |\psi^0_{(i,j)}|)$}
\AxiomC{$\Sigma$}
\BinaryInfC{$\perp$}	
\end{prooftree}
Analogously, given that $T^*_j \cup \left\{ \neg \phi_j \right\} \nvdash \perp$, we have that
\begin{prooftree}
\AxiomC{$T^*_j \cup \left\{ \neg \phi_j \right\}$}
\UnaryInfC{$\bigvee_{i < k_1} \bigwedge_{j < m_{(i, 1)}} \forall x(\phi^1_{(i,j)} \rightarrow \psi^1_{(i,j)})$}
\LeftLabel{($R_0$)}
\UnaryInfC{$\bigvee_{i < k_1} \bigwedge_{j < m_{(i, 1)}} (|\phi^1_{(i,j)}| \leq |\psi^1_{(i,j)}|)$}
\AxiomC{$\Sigma$}
\BinaryInfC{$\perp$}	
\end{prooftree}
But then 
\begin{prooftree}
\AxiomC{$T^*_j \cup \left\{ \phi_j \vee \neg \phi_j \right\}$}
\UnaryInfC{$\bigvee_{s < 2, \, i < k_s} \bigwedge_{j < m_{(i, s)}} \forall x(\phi^s_{(i,j)} \rightarrow \psi^s_{(i,j)})$}
\LeftLabel{($R_0$)}
\UnaryInfC{$\bigvee_{s < 2, \, i < k_s} \bigwedge_{j < m_{(i, s)}} (|\phi^s_{(i,j)}| \leq |\psi^s_{(i,j)}|)$}
\AxiomC{$\Sigma$}
\BinaryInfC{$\perp$}	
\end{prooftree}
which contradicts the fact that $(T^*_j, \Sigma) \nvdash \perp$. We now construct $\Sigma_0$. First of all, let $\Sigma'$ be the deductive closure of $\Sigma$ under the axioms $\mathrm{RCF}^*$ and the rule ($R_0$) with premises from $T_0$. Then $(T_0, \Sigma')$ must be consistent because otherwise there would be $i < \omega$ such that $(T^*_i, \Sigma)$ is not consistent. Now, just extend $\Sigma'$ to a complete $L_1$-theory $\Sigma_0$ using the Lindembaum's Lemma of first-order logic. Then $(T_0, \Sigma_0)$ is as wanted.
\end{proof}

	Before proving a completeness result, we need some elementary facts about elementary extensions of the ordered field of real numbers $\mathcal{R} = (\mathbb{R}, 0, 1, +, -, \cdot, \leq)$.
	Let $\mathcal{B}$ be an elementary extension of $\mathcal{R}$. We say that $b \in B$ is finite if there exist $r, s \in \mathbb{R}$ such that $r < b < s$. We denote by $\mathrm{Fin}(\mathcal{B})$ the set of finite elements of $\mathcal{B}$. Given $b \in \mathrm{Fin}(\mathcal{B})$ we denote by $\mathrm{st}(b)$ the standard part of $b$, i.e. $\mathrm{sup}(\{ r \in \mathbb{R} : r < b \})$ (for details see e.g. \cite[Section 5.6]{gold}). By positive bounded formulas we mean formulas which are built up from atomic formulas by means of conjunction $\wedge$, disjunction $\vee$, universal quantification $\forall$ and bounded existential quantification $\exists x(-n \leq  x \leq n \wedge \phi)$.
	
		\begin{fact}\label{positive_fr} Let $\mathcal{B}$ be an elementary extension of $\mathcal{R}$. Then, the map $\mathrm{st}: \mathrm{Fin}(\mathcal{B}) \rightarrow \mathbb{R}$ preserves positive bounded formulas.
\end{fact}

	\begin{proof} This can be proved by a straightforward induction on the complexity of positive formulas. Only the atomic case is interesting. For it see e.g. \cite[Theorem 6.7]{loeb} or \cite[Theorem 5.6.2]{gold} (in \cite{loeb} and \cite{gold} proofs are with respect to the hyperreals but they work for any elementary extension of $\mathcal{R}$). 
\end{proof}

	\begin{theorem}[Completeness]\label{main_theorem} Let $T$ be an $L_0$-theory and $\Sigma$ a positive bounded $L_1$-theory. Then the following are equivalent.
	\begin{enumerate}[(i)]
	\item There exists $\mathcal{A} \models T$ and a measure team $X = (\Omega, \mathcal{F}, P, \tau)$ with values in $\mathcal{A}$ and $\mathrm{dom}(X) = \left\{ v_i \, | \, i < n \right\}$, such that $X \models \Sigma$.
	\item $(T, \Sigma) \nvdash \perp$.
	\item As in (i), with $\Omega = [0, 1]$, $\mathcal{F}$ the $\sigma$-algebra $\mathcal{L}$ of Lebesgue measurable subsets of $[0, 1]$ and $P$ the Lebesgue measure on $[0, 1]$.
\end{enumerate}
\end{theorem}

\begin{proof} We only prove (ii) implies (iii). Suppose that $(T, \Sigma) \nvdash \perp$. By Lemma \ref{linden} we can find a complete $L_0$-theory $T_0$ and complete $L_1$-theory $\Sigma_0$ such that $T \subseteq T_0$, $\Sigma \subseteq \Sigma_0$ and $(T_0, \Sigma_0) \nvdash \perp$. In particular, the theories $T_0$ and $\Sigma_0$ are consistent (with respect to the deductive system of first-order logic), because otherwise we would be able to derive a contradiction also from our deductive system. Let then $\mathcal{B} \models \Sigma_0$. Given that our deductive system contains $\mathrm{RCF}^* = \mathrm{Th}(\mathcal{R}_Q)$,  we can---without loss of generality---assume that $\mathcal{R}_Q \preccurlyeq \mathcal{B} \restriction L_Q$. To see this, just take a sufficiently saturated elementary extension of $\mathcal{R}_Q$ and think of $\Sigma_0$ as a type of the theory $\mathrm{RCF}^* \subseteq \Sigma_0$ in the variables $|\phi(x)|$. We now expand $\mathcal{R}_Q$ to an $L_1$-structure by letting 
	$$ |\phi(x)|^{\mathcal{R}^{X}_Q} = \mathrm{st}(|\phi(x)|^{\mathcal{B}}) $$
for every $L_0$-formula $\phi(x)$ in the free variables $x = (v_i)_{i < n}$. By Fact \ref{positive_fr} we have  ${\mathcal{R}^{X}_Q} \models \Sigma$. We now want to define $\mathcal{A} \models T_0$ as well as  a measure team $X = ([0, 1], \mathcal{L}, P, \tau)$ with values in $\mathcal{A}$ and $\mathrm{dom}(X) = \left\{ v_i \, | \, i < n \right\}$, so that $$|\phi(x)|^{\mathcal{R}^{X}_Q} = [\phi(x)]_X$$ for every $L_0$-formula $\phi(x)$ in the free variables $x$. As to $\mathcal{A}$, we can let it be any $\omega$-saturated model of $T_0$ ($\omega$-compactness suffices if $n < \omega$). As to the team $X$, we do the following. Let $(\phi_i(x))_{i < \omega}$ be an enumeration of the $L_0$-formulas in the free variables $x$ (remember that $x$ is a vector). We label $2^{< \omega}$ with subsets $I_{\sigma} \subseteq [0, 1)$ as in Figure \ref{tree} (where for simplicity we let $|\phi| = |\phi(x)|^{\mathcal{R}^{X}_Q}$, $\phi_{(1, 1)} = \phi_0 \wedge \phi_1$ and $\phi_{(0, 1)} = \neg\phi_0 \wedge \phi_1$).

\begin{figure}[ht]
	\begin{center}
	\begin{tikzpicture}[level distance=1.5cm,
  level 1/.style={sibling distance=5.85cm},
  level 2/.style={sibling distance=3cm}]
  level 3/.style={sibling distance=0.1cm}]
  \node {$\emptyset$}
    child {node {$[0, |\phi_0|)$}
    child {node {$[0, |\phi_{(1, 1)}|)$}
    child {node {$\cdots$}}
    child {node {$\cdots$}}}
    child {node {$[|\phi_{(1, 1)}|, |\phi_0|)$}
    child {node {$\cdots$}}
    child {node {$\cdots$}}}
    }
    child {node {$[|\phi_0|, 1)$}
    child {node {$[|\phi_0|, |\phi_0| + |\phi_{(0, 1)}|)$}
    child {node {$\cdots$}}
    child {node {$\cdots$}}}
    child {node {$[|\phi_0| + |\phi_{(0, 1)}|, 1)$}
    child {node {$\cdots$}}
    child {node {$\cdots$}}}
    };
\end{tikzpicture}
\end{center}  \caption{Labelling $2^{< \omega}$ with $I_{\sigma} \subseteq [0, 1)$}\label{tree}
\end{figure}
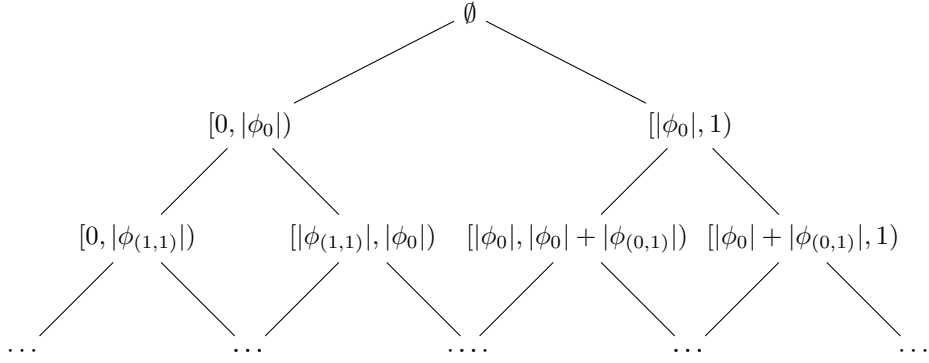
Because of ($A_0$)-($A_2)$ and $(R_0)$, given $s \in [0, 1)$ and $1 \leq i < j < \omega$, we have $s \in I_{\sigma_i} \wedge I_{\sigma_j}$ for unique $\sigma_i \in 2^i$ and $\sigma_j \in 2^j$, and moreover $\sigma_i \subseteq \sigma_j$. Thus, to every $s \in [0, 1)$ we can associate 
	$$f_s = \bigcup_{1 \leq i < \omega}\sigma_i \in 2^{\omega}.$$
Let 
	$$ \mathrm{tp}({f_s}) = \mbox{\Large $\{$ } \!\!\! \bigwedge_{i < m} \phi_i^{f_s(i)}(x) \, | \,1 \leq m < \omega \mbox{\Large $\}$.}$$
Then, because of ($A_0$) and $(R_0)$, the type $\mathrm{tp}({f_s})$ is finitely satisfiable and hence satisfiable in $\mathcal{A}$ ($\omega$-saturated models realize types over the empty set in infinitely many variables). Let then $a_s \in A$ such that $a_s \models \mathrm{tp}({f_s})$ and 
	\[ \tau(s)(x) = \begin{cases} a \in A^{|x|} \;\;\;\;\; \text{ if } s = 1 \\
						  a_s \;\;\;\;\;\;\;\;\;\;\;\;\;\; \text{ if } s \in [0, 1). \end{cases} \]
Then $X = ([0, 1], \mathcal{L}, P, \tau)$ is as desired.
\end{proof}

	The following standard counterexample shows that the positivity of $\Sigma$ is a necessary condition in Theorem \ref{main_theorem}.
	
	\begin{example} Let $L_0$ consists of a single unary predicate $R$, $T$ be the empty theory and
	\[ \Sigma = \left\{ |R(x)| > 0 \right\} \cup \mbox{\Large $\{$ } \!\!\! |R(x)| \leq \frac{1}{n} \, | \, 0 < n < \omega \mbox{\Large $\}$.}\]
	Then $(T, \Sigma) \nvdash \perp$, because $\Sigma$ is finitely satisfiable, but (i) of Theorem \ref{main_theorem} fails, as in fact there is no way to expand $\mathcal{R}_Q$ to an $L_1$-structure $\mathcal{R}^X_Q$ so that $\mathcal{R}^X_Q \models \Sigma$.
\end{example}

	On the other hand, if we insist on $\Sigma$ being finite we can prove Theorem \ref{main_theorem} without the positive bounded assumption.

	\begin{theorem} As in Theorem \ref{main_theorem} with $\Sigma$ finite and arbitrary, i.e. not necessarily positive bounded. 
\end{theorem}

\begin{proof} The proof is essentially as in Theorem \ref{main_theorem}. We only have to specify how to define $|\phi(x)|^{\mathcal{R}^{X}_Q}$ in this case. We do this. Let $\Sigma_0$ and $\mathcal{B} \models \Sigma_0$ as in the proof of Theorem \ref{main_theorem}. First of all, extend $\Sigma$ to a $\Sigma'$ adding the following formulas for every $|\phi(x)|$ and $|\psi(x)|$ occurring in $\Sigma$:
\begin{enumerate}[(i)]
\item $0 \leq |\phi(x)| \leq 1$;
\item $|\neg \phi(x)| = 1 - |\phi(x)|$;
\item $|\phi(x)| = |\phi(x) \wedge \psi(x)| + |\phi(x) \wedge \neg \psi(x)|$;
\item $|\phi(x) \vee \psi(x)| = |\phi(x)| + |\psi(x)| - |\phi(x) \wedge \psi(x)|$.
\end{enumerate} 
Further extend $\Sigma'$ to a $\Sigma''$ requiring that if $|\phi(x)| = 0 \in \Sigma_0$ and $|\phi(x)|$ occurs in $\Sigma$, then $|\phi(x)| = 0 \in \Sigma''$.
Items (i)-(iv) are theorems of our logic, and so $\Sigma'' \subseteq \Sigma_0$.
Secondly, notice that it suffices to specify $|\phi(x)|^{\mathcal{R}^{X}_Q}$ only for the $|\phi(x)|$ occurring in $\Sigma''$, but this is easily done---just consider $\bigwedge \Sigma''$, substitute constants of the form $|\phi(x)|$ with free variables, quantify existentially and find a real solution using the fact that $\mathcal{R}_Q \preccurlyeq \mathcal{B} \restriction L_Q$. The rest of the proof is clear (in this case enumerate only the $L_0$-formulas that occur in $\Sigma$ and construct a finite tree).
\end{proof}

	\begin{corollary} Let $T$ be an $L_0$-theory and $\Sigma \cup \left\{ \alpha \right\}$ a finite $L_1$-theory. Then
	\[ (T, \Sigma) \vdash \alpha \;\; \Leftrightarrow \;\; (T, \Sigma) \models \alpha \]
\end{corollary}

	The main source of inspiration for our logic is of course for $T = T_0 = \mathrm{Th}(\mathcal{A})$, with $\mathcal{A}$ a particular $L_0$-structure. If we wish the class of teams with values in $\mathcal{A}$ to be complete in the sense of providing every possible counterexample for $(T, \Sigma) \nvdash \perp$ (as in Theorem \ref{main_theorem}), then we have to require $\omega$-compactness or $\omega$-saturation (depending on whether $n < \omega$ or $n = \omega$) of $\mathcal{A}$. If $\mathcal{A}$ is finite then, of course, we do not have this problem (since it is $\omega$-saturated). In particular, for $L_0$ the signature of boolean algebras and $\mathcal{A}$ the boolean algebra $\left\{ 0, 1 \right\}$, we have a system of propositional probability logic properly extending the probability logic considered in \cite{QTL}. 



\end{document}